\documentclass[11pt]{article}
\usepackage{natbib}
\usepackage{times}
\usepackage[latin1]{inputenc}
\usepackage{amsmath,amssymb}
\usepackage{color}
\usepackage{verbatim}
\usepackage{float}
\usepackage{graphics}
\usepackage{graphicx}
\usepackage[ruled,vlined]{algorithm2e}
\usepackage{caption}
\newcommand{\ve}[2]{\langle #1 ,  #2 \rangle}
\newcommand{\St}[2]{\mathcal{S}^{#2}_{#1}} 

\newcommand{\sparsity}{\gamma}
\newcommand{\level}{\omega}

\newcommand{\single}{\phi_{\ell_1}} 
\newcommand{\singlec}{\phi_{\ell_0}} 
\newcommand{\block}{\phi_{\ell_1,m}} 
\newcommand{\blockc}{\phi_{\ell_0,m}} 

\newcommand{\Asingle}{\mathsf{GPower}_{\ell_1}}
\newcommand{\Asinglec}{\mathsf{GPower}_{\ell_0}}
\newcommand{\Ablock}{\mathsf{GPower}_{\ell_1,m}}
\newcommand{\Ablockc}{\mathsf{GPower}_{\ell_0,m}}

\newcommand{\rSVD}{\mathsf{rSVD}_{\ell_1}}
\newcommand{\rSVDc}{\mathsf{rSVD}_{\ell_0}}
\newcommand{\Greedy}{\mathsf{Greedy}}
\newcommand{\SPCA}{\mathsf{SPCA}}
\newcommand{\DSPCA}{\mathsf{DSPCA}}
\newcommand{\PCA}{\mathsf{PCA}}

\newcommand{\sphere}[1]{\mathcal{S}^{#1}}
\newcommand{\sphereprod}[2]{[\mathcal{S}^{#1}]^{#2}}
\newcommand{\sph}{\mathcal{S}}
\newcommand{\ball}{\mathcal{B}}

\newcommand{\Q}{\mathcal{Q}}
\newcommand{\R}{\mathbf{R}}
\newcommand{\E}{\mathbf{E}}
\newcommand{\Scone}{\mathbf{S}}

\newcommand{\G}{G}

\newcommand{\card}[1]{\|#1\|_{0}}

\DeclareMathOperator{\conv}{Conv}
\DeclareMathOperator{\Diag}{Diag}
\DeclareMathOperator{\trace}{Tr}
\DeclareMathOperator{\rank}{Rank}
\DeclareMathOperator{\uf}{Uf}
\DeclareMathOperator{\signum}{sign}
\DeclareMathOperator{\Arg}{Arg}

\DeclareMathOperator{\Var}{Var}
\DeclareMathOperator{\AVar}{AdjVar}

\newcommand{\eqdef}{\stackrel{\text{def}}{=}}
\newcommand{\indef}{\stackrel{\text{def}}{\in}}
\newcommand{\suchthat}{\;|\;}

\newcommand{\combib}[1]{}

\newcommand{\tabsize}{\small}
\newcommand{\capsize}{\footnotesize}
\newcommand{\figwidth}{7cm}

\newenvironment{proof}{{\it Proof.}}{\hspace{\stretch{1}} $\square$}
\newtheorem{thrm}{Theorem}
\newtheorem{lmm}[thrm]{Lemma}

\newtheorem{prpstn}[thrm]{Proposition}

\newtheorem{xmpl}[thrm]{Example}

\newtheorem{ssmptn}{Assumption}

\begin{document}

\title{Generalized power method\\ for sparse principal component analysis}
\author{Michel Journée\footnotemark[1] \and Yurii Nesterov\footnotemark[2] \and Peter Richt\'{a}rik\footnotemark[2] \and Rodolphe Sepulchre\footnotemark[1]}
\date{}
\renewcommand{\thefootnote}{\fnsymbol{footnote}}
 \footnotetext[1]{Department of Electrical
Engineering and Computer Science, University of Li\`{e}ge, B-4000
Liège, Belgium. Email: [M.Journee, R.Sepulchre]@ulg.ac.be}
\footnotetext[2]{Center for Operations Research and Econometrics, Catholic University of Louvain, Voie du Roman Pays 34, B-1348 Louvain-la-Neuve, Belgium. Email: [Nesterov, Richtarik]@core.ucl.ac.be}
\maketitle
\renewcommand{\thefootnote}{\arabic{footnote}}

\begin{abstract}
In this paper we develop a new approach to sparse principal component analysis (sparse PCA). We propose two single-unit and two block optimization formulations of the sparse PCA problem, aimed at extracting a single sparse dominant principal component of a data matrix, or more components at once, respectively. While the initial formulations involve nonconvex functions, and are therefore computationally intractable, we rewrite them into the form of an optimization program involving maximization of a convex function on a compact set. The dimension of the search space is decreased enormously if the data matrix has many more columns (variables) than rows. We then propose and analyze a simple gradient method suited for the task. It appears that our algorithm has best convergence properties in the case when either the objective function or the feasible set are strongly convex, which is the case with our single-unit formulations and can be enforced in the block case. Finally, we demonstrate numerically on a set of random and gene expression test problems that our approach outperforms existing algorithms both in quality of the obtained solution and in computational speed.
\end{abstract}

\section{Introduction}

\emph{Principal component analysis} (PCA) is a well established tool for making sense of high dimensional data by reducing it to a smaller dimension.
It has applications virtually in all areas of science---machine learning, image processing, engineering, genetics, neurocomputing, chemistry, meteorology, control theory, computer networks---to name just a few---where large data sets are encountered. It is important that having reduced dimension, the essential characteristics of the data are retained. If  $A \in \R^{p \times n}$ is a matrix encoding $p$ samples of $n$ variables, with $n$ being large, PCA aims at finding a few linear combinations of these
variables, called \emph{principal components}, which point in orthogonal directions explaining as much of the variance in the data as possible. If the variables contained in the columns of $A$ are centered, then the classical PCA can be written in terms of the scaled \emph{sample covariance matrix} $\Sigma = AA^T$ as follows:
\begin{equation}\label{eq:PCA}\text{Find} \quad z^* =\arg \max_{z^Tz \leq 1} z^T \Sigma z. \end{equation}

Extracting one component amounts to computing the dominant eigenvector of $\Sigma$ (or, equivalently, dominant right singular vector of $A$). Full PCA involves the computation of the singular value decomposition (SVD) of $A$.
Principal components are, in general, combinations of all the input variables, i.e. the \emph{loading vector} $z^*$ is not expected to have many zero coefficients. In most applications, however, the original variables have concrete physical meaning and PCA then appears especially interpretable if
the extracted components are composed only from a small number of the original variables. In the case of gene expression data, for instance, each variable
represents the expression level of a particular gene. A good analysis tool  for biological interpretation should be capable to highlight ``simple'' structures
in the genome---structures expected to involve a few genes only---that explain a significant amount of the specific biological
processes encoded in the data. Components that are linear combinations of a small number of variables are, quite naturally, usually easier to interpret.
It is clear, however, that with this additional goal, some of the explained variance has to be sacrificed. The objective of
\emph{sparse principal component analysis} (sparse PCA) is to find a reasonable \emph{trade-off} between these conflicting goals. One would like to
explain \emph{as much} variability in the data as possible, using components constructed from \emph{as few} variables as possible. This is the classical trade-off between \emph{statistical fidelity}  and \emph{interpretability}.

\bigskip
For about a decade, sparse PCA has been a topic of active research.
Historically, the first suggested approaches  were based on ad-hoc methods involving post-processing of the components obtained from classical PCA.
For example, \combib{Joliffe et al. }\citet{jolliffe95} considers using various rotation
techniques to find sparse loading vectors in the subspace identified by
PCA. \combib{Cadima et al. }\citet{Cadima95} propose to simply set to zero the PCA loadings
which are in absolute value smaller than some threshold constant.

In recent years, more involved approaches have been put forward---approaches that consider
the conflicting goals of explaining variability and achieving representation sparsity simultaneously. These
methods usually cast the sparse PCA problem in the form of an optimization program, aiming at maximizing explained variance penalized for
the number of non-zero loadings. For instance, the SCoTLASS algorithm proposed by \combib{Jolliffe et al. }\citet{jolliffe03} aims at maximizing the Rayleigh quotient
of the covariance matrix of the data under the $\ell_1$-norm based Lasso penalty (\citet{Tibshirani96}). \combib{Zou et al. }\citet{Zou04} formulate sparse PCA as a
regression-type optimization problem and impose the Lasso penalty on the regression coefficients. \combib{d'Aspremont et al. }\citet{Aspremont07} in their $\DSPCA$ algorithm exploit
convex optimization tools to solve a convex relaxation of the sparse PCA problem.
\combib{Shen and Huang }\citet{shen08} adapt the singular value decomposition (SVD) to compute low-rank matrix approximations of the data matrix
under various sparsity-inducing penalties. Greedy methods, which are typical for combinatorial problems, have been investigated by \combib{Moghaddam et al. }\citet{Moghaddam06}. Finally,  \combib{d'Aspremont et al. }\citet{Aspremont07b} propose a greedy heuristic accompanied with a certificate of optimality.

\bigskip
In many applications, several components need to be identified. The traditional approach consists of incorporating an existing single-unit algorithm in a deflation scheme, and computing the desired number of components sequentially (see, e.g., \combib{d'Aspremont et al. }\citet{Aspremont07}). In the case of Rayleigh quotient maximization it is well-known that computing several components at once instead of computing them one-by-one by deflation with the classical power method might present better convergence whenever the largest eigenvalues of the underlying matrix are close to each other (see, e.g., \combib{Parlett }\citet{Parlett80}). Therefore, block approaches for sparse PCA are expected to be more efficient on ill-posed problems.

\bigskip
In this paper we consider two single-unit (Section 2.1 and 2.3) and two block formulations (Section 2.3 and 2.4) of sparse PCA, aimed at extracting $m$ sparse principal components, with $m=1$ in the former case and $p\geq m>1$ in the latter. Each of these two groups comes in two variants, depending on the type of penalty we use to enforce sparsity---either $\ell_1$ or $\ell_0$ (cardinality).~\footnote{Our single-unit cardinality-penalized formulation is identical to that of \combib{d'Aspremont et al. }\citet{Aspremont07b}.} While our basic formulations involve maximization of a \emph{nonconvex} function on a space of dimension involving $n$, we construct \emph{reformulations} that cast the problem into the form of maximization of a \emph{convex} function on the unit Euclidean sphere in $\R^p$ (in the $m=1$ case) or the \emph{Stiefel manifold}\footnote{Stiefel manifold is the set of rectangular matrices with orthonormal columns.} in $\R^{p\times m}$ (in the $m>1$ case). The advantage of the reformulation becomes apparent when trying to solve problems with many variables ($n\gg p$), since we manage to avoid searching a space of large dimension. At the same time, due to the convexity of the new cost function we are able to propose and \emph{analyze} the iteration-complexity of a simple gradient-type scheme, which appears to be well suited for problems of this form. In particular, we study (Section \ref{sec:theory}) a first-order method for solving an optimization problem of the form
\begin{equation}\tag{P}
f^* = \max_{x\in \Q} f(x),
\end{equation}
where $\Q$ is a compact subset of a finite-dimensional vector space and $f$ is convex. It appears that our method has best theoretical convergence properties when either $f$ or $\Q$ are strongly convex, which is the case in the single unit case (unit ball is strongly convex) and can be enforced in the block case by adding a strongly convex regularizing term to the objective function, constant on the feasible set. We do not, however, prove any results concerning the quality of the obtained solution. Even the goal of obtaining a local maximizer is in general unattainable, and we must be content with convergence to a stationary point.

In the particular case when $\Q$ is the unit Euclidean ball in $\R^p$ and $f(x) = x^T C x$ for some $p\times p$ symmetric positive definite matrix $C$, our gradient scheme specializes to the \emph{power method}, which aims at maximizing the \emph{Rayleigh quotient} \[R(x)=\frac{x^T C x}{x^T x}\] and thus at computing the largest eigenvalue, and the corresponding eigenvector, of $C$.

\bigskip
By applying our general gradient scheme to our sparse PCA reformulations of the form (P), we obtain algorithms (Section 4) with per-iteration computational cost $\mathcal{O}(npm)$.

\bigskip
We demonstrate on random Gaussian (Section 5.1) and gene expression data related to breast cancer (Section 5.2) that our methods are very efficient in practice. While achieving a balance between the explained variance and sparsity which is the same as or superior to the existing methods, they are faster, often converging before some of the other algorithms manage to initialize. Additionally, in the case of gene expression data our approach seems to extract components with strongest biological content.

\bigskip
\textbf{Notation.}  For convenience of the reader, and at the expense of redundancy, some of the less standard notation below is also introduced at the appropriate place in the text where it is used. Parameters $m\leq p\leq n$ are actual values of dimensions of spaces used in the paper. In the definitions below, we use these actual values (i.e. $n,p$ and $m$) if the corresponding object we define is used in the text exclusively with them; otherwise we make use of the dummy variables $k$ (representing $p$ or $n$ in the text) and $l$ (representing $m,p$ or $n$ in the text).

We will work with vectors and matrices of various sizes ($\R^k, \R^{k\times l}$). Given a vector $y \in \R^k$, its $i^{\mathrm{th}}$ coordinate is denoted by $y_i$.
For a matrix $Y \in \R^{k \times l}$, $y_i$ is the $i^{\mathrm{th}}$ column of $Y$ and $y_{ij}$ is the element of $Y$ at position $(i,j)$.

By $\E$ we refer to a finite-dimensional vector space; $\E^*$ is its conjugate space, i.e. the space of all linear functionals on $\E$. By $\ve{s}{x}$ we denote the action of $s\in\E^*$ on $x\in \E$. For a self-adjoint positive definite linear operator $\G:\E\to \E^*$ we define a pair of norms on $\E$ and $\E^*$ as follows
\begin{equation}\label{eq:norms}
\begin{array}{rcl}
\|x\| & \eqdef & \ve{\G x}{x}^{1/2}, \quad x\in \E,\\
\\
\|s\|_* & \eqdef & \ve{s}{\G^{-1}s}^{1/2}, \quad s\in \E^*.
\end{array}
\end{equation}

Although the theory in Section~\ref{sec:theory} is developed in this general setting, the sparse PCA applications considered in this paper require either the choice $\E=\E^*=\R^p$ (see Section~\ref{sec:vector_setting} and problems (\ref{eq:SPCA3}) and (\ref{eq:SPCA4}) in Section~\ref{sec:spca}) or $\E=\E^*=\R^{p\times m}$ (see Section~\ref{sec:matrix_setting} and problems (\ref{eq:block_spca2}) and (\ref{eq:block_spca4}) in Section~\ref{sec:spca}). In both cases we will let $\G$ be the corresponding identity operator for which we obtain
\[\ve{x}{y} = \sum_i x_i y_i, \quad \|x\| = \ve{x}{x}^{1/2} = \left(\sum_i x_i^2\right)^{1/2} = \|x\|_2, \quad x,y\in \R^p, \text {and}\]
\[\ve{X}{Y} = \trace X^TY, \quad \|X\| = \ve{X}{X}^{1/2} = \left(\sum_{ij} x_{ij}^2\right)^{1/2} = \|X\|_F, \quad X,Y\in \R^{p\times m}.\]

Thus in the vector setting we work with the \emph{standard Euclidean norm} and in the matrix setting with the \emph{Frobenius norm}. The symbol $\trace$ denotes the trace of its argument.

Furthermore, for $z\in \R^n$ we write $\|z\|_1=\sum_i |z_i|$ ($\ell_1$ norm) and by $\|z\|_0$ ($\ell_0$ ``norm") we refer to the number of nonzero coefficients, or \emph{cardinality}, of $z$. By $\Scone^p$ we refer to the space of
all $p\times p$ symmetric matrices; $\Scone^p_+$ (resp.
$\Scone^p_{++}$) refers to the positive semidefinite (resp.
definite) cone. Eigenvalues of matrix $Y$ are denoted by $\lambda_i(Y)$, largest eigenvalue by $\lambda_{\text{max}}(Y)$. Analogous notation with the symbol $\sigma$ refers to singular values.

By $\ball^{k}=\{y \in \R^k \suchthat y^T y \leq 1\}$ (resp. $\sphere{k}~ = \{y \in \R^k \suchthat y^T y =
1\}$) we refer
to the unit Euclidean ball (resp. sphere) in $\R^k$. If we write $\ball{}$ and $\sphere{}$, then these are the corresponding
objects in $\E$. The space of $n\times m$ matrices with unit-norm columns will be denoted by
\[\sphereprod{n}{m}  = \{Y \in \R^{n\times m}
\suchthat \Diag(Y^T Y) = I_m\},\]
where $\Diag(\cdot)$ represents the diagonal matrix obtained by extracting the diagonal of the argument.
\emph{Stiefel manifold} is the set of rectangular matrices of fixed size with orthonormal columns: \[\St{m}{p} = \{Y \in \R^{p\times
m} \suchthat Y^T Y = I_m\}.\]

 For $t\in \R$ we will further write $\signum(t)$ for the sign of the argument and $t_+ = \max\{0,t\}$.

\bigskip
\section{Some formulations of the sparse PCA problem} \label{sec:spca}

In this section we propose four formulations of the sparse PCA
problem, all in the form of the general optimization framework (P).
The first two deal with the single-unit sparse PCA problem and the remaining two are their generalizations to the block case.

\subsection{Single-unit sparse PCA via $\ell_1$-penalty}

Let us consider the optimization problem
\begin{equation}
\label{eq:RQ_l1} \single(\sparsity) \eqdef \max_{z \in \ball^n}
\sqrt{z^T \Sigma z} - \sparsity \|z\|_1,
\end{equation}
with sparsity-controlling parameter $\sparsity \geq 0$ and sample covariance matrix $\Sigma = A^T A$.

The solution $z^*(\sparsity)$ of (\ref{eq:RQ_l1}) in the case
$\sparsity=0$ is equal to the right singular vector corresponding to $\sigma_{\text{max}}(A)$,
the largest singular value of $A$. It is the first principal
component of the data matrix $A$. The optimal value of the problem
is thus equal to
\[\single(0) = (\lambda_{\text{max}}(A^T A))^{1/2} = \sigma_{\text{max}}(A).\]
Note that there is no reason to expect this vector to be sparse. On
the other hand, for large enough $\sparsity$, we will necessarily
have $z^*(\sparsity)=0$, obtaining maximal sparsity. Indeed, since
\[\max_{z\neq 0} \frac{\|Az\|_2}{\|z\|_1} = \max_{z\neq 0} \frac{\|\sum_i z_i a_i\|_2}{\|z\|_1} \leq \max_{z \neq 0} \frac{\sum_i |z_i|\|a_i\|_2}{\sum_i |z_i|} = \max_i \|a_i\|_2 = \|a_{i^*}\|_2,\]
we get $\|Az\|_2 - \sparsity \|z\|_1 < 0$ for all nonzero vectors $z$
whenever $\sparsity$ is chosen to be strictly bigger than $\|a_{i^*}\|_2$. From now on we will assume that
\begin{equation}\label{eq:2:bound_on_sparsity}\sparsity < \|a_{i^*}\|_2.\end{equation}

Note that there is a trade-off between the value
$\|Az^*(\sparsity)\|_2$ and the sparsity of the solution
$z^*(\sparsity)$. The penalty parameter $\sparsity$ is introduced to
``continuously" interpolate between the two extreme cases described
above, with values in the interval $[0,\|a_{i^*}\|_2)$. It depends on the particular application whether sparsity is
valued more than the explained variance, or vice versa, and to what
extent. Due to these considerations, we will consider the solution
of (\ref{eq:RQ_l1}) to be a sparse principal component of $A$.

\bigskip
\textbf{Reformulation.} The reader will observe that the objective function in (\ref{eq:RQ_l1}) is not convex, nor concave, and that the feasible set is of a high dimension if $p\ll n$.
It turns out that these shortcomings are overcome by considering the following reformulation:
\begin{align}
\single(\sparsity)&= \max_{z \in \ball^n} \|Az\|_2 - \sparsity \|z\|_1 \nonumber\\
&=\max_{z \in \ball^n} \max_{x \in \ball^p} x^T A z - \sparsity \|z\|_1 \label{eq:SPCA_single_unit}\\
&=\max_{x \in \ball^p} \max_{z \in \ball^n} \sum_{i=1}^n z_i (a_i^T x) - \sparsity |z_i| \nonumber \\
&=\max_{x \in \ball^p}\max_{\bar{z} \in \ball^n} \sum_{i=1}^n |\bar{z}_i| (|a_i^T x| - \sparsity),
\label{eq:SPCA}
\end{align}
where $z_i = \signum(a_i^Tx)\bar{z}_i$. In view of
(\ref{eq:2:bound_on_sparsity}), there is some  $x\in \ball^n$ for
which $a_i^Tx > \sparsity$. Fixing such $x$, solving the inner
maximization problem for $\bar{z}$ and then translating back to $z$,
we obtain the closed-form solution  \begin{equation}
\label{eq:close_form_sol_x} z_i^* = z_i^*(\sparsity) = \frac{\signum(a_i^T x) [|a_i^T
x| - \sparsity]_+}{\sqrt{\sum_{k=1}^n [|a_k^T x| - \sparsity]_+^2}}, \quad i=1,\dots,n.
\end{equation}
Problem (\ref{eq:SPCA}) can therefore be written in the form
\begin{equation} \label{eq:SPCA3} \boxed{
\single^2(\sparsity)=\max_{x \in \sphere{p}} \sum_{i=1}^n
[|a_i^T x| - \sparsity]_+^2.}
\end{equation}
Note that the objective function is differentiable and convex, and hence all local and global maxima must lie on the
boundary, i.e., on the unit Euclidean sphere $\sphere{p}$.
Also, in the case when $p \ll n$, formulation (\ref{eq:SPCA3})
requires to search a space of a much lower dimension than the
initial problem (\ref{eq:RQ_l1}).

\bigskip
\textbf{Sparsity.} In view of (\ref{eq:close_form_sol_x}), an optimal solution $x^*$ of
(\ref{eq:SPCA3}) defines a sparsity pattern of the vector $z^*$. In
fact, the coefficients of $z^*$ indexed by
\begin{equation}
\label{eq:sparsity_pattern_l1}
\mathcal{I}=\{i \suchthat |a_i^T x^*|
> \sparsity \}\end{equation} are active while all others must
be zero. Geometrically, active indices correspond to the defining
hyperplanes of the polytope \[\mathcal{D} = \{x \in \R^{p} \suchthat
|a_i^T x| \leq 1\}\] that are (strictly) crossed by the line joining
the origin and the point $x^*/\sparsity$. Note that it is possible to say something about the sparsity of the solution even without the knowledge of $x^*$:
\begin{equation}\label{eq:singleL1_sparsity}\sparsity \geq \|a_{i}\|_2 \quad \Rightarrow  \quad z^*_i(\sparsity) = 0, \qquad i=1,\dots,n.\end{equation}

\subsection{Single-unit sparse PCA via cardinality penalty}
\label{sec:spca_single_card} Instead of the $\ell_1$-penalization,
\combib{the authors of }\citet{Aspremont07b} consider the formulation
\begin{equation}
\label{eq:RQ_l0} \singlec(\sparsity) \eqdef \max_{z \in
\ball^n} z^T \Sigma z - \sparsity \; \card{z},
\end{equation}
which directly penalizes the number of nonzero components
(cardinality) of the vector $z$.

\bigskip
\textbf{Reformulation.} The reasoning of the previous section suggests the reformulation
\begin{equation}
\label{eq:spca_single_l0} \singlec(\sparsity)= \max_{x \in
\ball^p} \max_{z \in \ball^n}  (x^T A z)^2 - \sparsity
\card{z},\end{equation} where the maximization with respect to $z
\in \ball^n$ for a fixed  $x \in \ball^p$ has the closed form solution
\begin{equation}
\label{eq:spca_single_l0_2} z_i^* = z_i^*(\sparsity)=\frac{[\signum((a_i^T
x)^2-\sparsity)]_+ a_i^T x}{\sqrt{\sum_{k=1}^n [\signum((a_k^T
x)^2-\sparsity)]_+ (a_k^T x)^2}}, \quad i=1,\dots,n.
\end{equation}
In analogy with the $\ell_1$ case, this derivation assumes that
\[\sparsity < \|a_{i^*}\|_2^2,\]
so that there is $x\in \ball^n$ such that $(a_i^T x)^2-\sparsity>0$. Otherwise $z^*=0$ is optimal.
Formula (\ref{eq:spca_single_l0_2}) is easily obtained by analyzing (\ref{eq:spca_single_l0}) separately for fixed cardinality values of $z$.
Hence, problem (\ref{eq:RQ_l0}) can be cast in the following form
\begin{equation}
\label{eq:SPCA4}\boxed{\singlec(\sparsity)=\max_{x \in \sphere{p}} \sum_{i=1}^n [(a_i^T x)^2 - \sparsity]_+.}
\end{equation}
Again, the objective function is convex, albeit nonsmooth, and the new search space is of particular interest if $p\ll n$. A different derivation of (\ref{eq:SPCA4}) for the $n=p$ case can be found
in \citet{Aspremont07b}.

\bigskip
\textbf{Sparsity.} Given a solution
$x^*$ of (\ref{eq:SPCA4}), the set of active indices of $z^*$ is
given by
\[\mathcal{I}=\{i \suchthat (a_i^T x^*)^2 >\sparsity \}.\]
Geometrically, active indices correspond to the defining
hyperplanes of the polytope \[\mathcal{D} = \{x \in \R^{p} \suchthat
|a_i^T x| \leq 1\}\] that are (strictly) crossed by the line joining
the origin and the point $x^*/\sqrt{\sparsity}$. As in the $\ell_1$ case,
we have
\begin{equation}\label{eq:singleL0_sparsity}\sparsity \geq \|a_{i}\|_2^2 \quad \Rightarrow  \quad z^*_i(\sparsity) = 0, \qquad i=1,\dots,n.\end{equation}

\subsection{Block sparse PCA via $\ell_1$-penalty}
\label{sec:block_spca_l1}
Consider the following block generalization of (\ref{eq:SPCA_single_unit}),
\begin{equation}
\block(\sparsity) \eqdef \max_{\substack{X \in \St{m}{p}\\Z \in \sphereprod{n}{m}}} \trace(X^T A Z N)
- \sparsity \sum_{j=1}^m \sum_{i=1}^n |z_{ij}| \label{eq:block_spca1},
\end{equation}
where $\sparsity \geq 0$ is a sparsity-controlling parameter and $N=\Diag(\mu_1, \ldots, \mu_m)$, with positive entries on the diagonal. The dimension $m$ corresponds to the number of extracted components and is assumed to be smaller or equal to the rank of the data matrix, i.e., $m \leq \rank(A)$. It will be shown below that under some conditions on the parameters $\mu_i$, the case $\sparsity=0$ recovers PCA. In that particular instance, any solution $Z^*$ of (\ref{eq:block_spca1}) has orthonormal columns, although this is not explicitly enforced. For positive $\sparsity$, the columns of $Z^*$ are not expected to be orthogonal anymore. Most existing algorithms for computing several sparse principal components, e.g., \citet{Zou04,Aspremont07,shen08}, also do not impose orthogonal loading directions. Simultaneously enforcing sparsity and orthogonality seems to be a hard (and perhaps questionable) task.

\bigskip
\textbf{Reformulation.} Since problem (\ref{eq:block_spca1}) is completely decoupled in the columns of
$Z$, i.e.,
\begin{equation*}
 \block(\sparsity) = \max_{X \in \St{m}{p}} \sum_{j=1}^m  \max_{z_j \in \sphere{n}} \mu_j x_j^T A z_j
-\sparsity \|z_{j}\|_1,
\end{equation*}
the closed-form solution (\ref{eq:close_form_sol_x}) of (\ref{eq:SPCA_single_unit}) is easily
adapted to the block formulation (\ref{eq:block_spca1}):
\begin{equation}\label{eq:spca_block_l1x}
 z_{ij}^* = z_{ij}^*(\sparsity) = \frac{\signum(a_i^T x_j)
[\mu_j |a_i^T x_j| - \sparsity]_+}{\sqrt{\sum_{k=1}^n [\mu_j |a_k^T x_j| -
\sparsity]_+^2}}.
\end{equation}
This leads to the reformulation
\begin{equation}
\label{eq:block_spca2} \boxed{\block^2(\sparsity)=
\underset{X \in \St{m}{p}}{\max} \sum_{j=1}^m \sum_{i=1}^n
[\mu_j |a_i^T x_j|-\sparsity]_+^2,}
\end{equation}
which maximizes a convex function $f: \R^{p \times m} \rightarrow \R$ on the Stiefel manifold $\St{m}{p}$.

\bigskip
\textbf{Sparsity.} A solution $X^*$ of (\ref{eq:block_spca2}) again
defines the sparsity pattern of the matrix $Z^*$: the entry $z^*_{ij}$
is active if \[\mu_j |a_i^T x_j^*|
> \sparsity,\]
and equal to zero otherwise. For
$\sparsity~>~\max_{i,j}  \mu_j \|a_i\|_2$, the trivial solution $Z^*=0$ is optimal.

\bigskip
\textbf{Block PCA.} For $\sparsity=0$, problem (\ref{eq:block_spca2}) can be equivalently written in the form
\begin{equation}
\label{eq:brockett} \block^2(0) = \max_{X \in \St{m}{p}} \trace(X^T A A^T X N^2),
\end{equation}
which has been well studied (see e.g., \combib{Brockett }\citet{Brockett91} and \combib{Absil et al. }\citet{AbsMahSep2008}). The solutions of (\ref{eq:brockett}) span the dominant $m$-dimensional invariant subspace of the matrix $A A^T$. Furthermore, if the parameters $\mu_i$ are all distinct, the columns of $X^*$ are the $m$ dominant eigenvectors of $A A^T$, i.e., the $m$ dominant left-eigenvectors of the data matrix $A$. The columns of the solution $Z^*$ of (\ref{eq:block_spca1}) are thus the $m$ dominant right singular vectors of $A$, i.e., the PCA loading vectors. Such a matrix $N$ with distinct diagonal elements enforces the objective function in (\ref{eq:brockett}) to have isolated maximizers. In fact, if $N=I_m$, any point $X^* U$ with $X^*$ a solution of (\ref{eq:brockett}) and $U \in \St{m}{m}$ is also a solution of (\ref{eq:brockett}). In the case of sparse PCA, i.e., $\sparsity>0$, the penalty term enforces isolated maximizers. The technical parameter $N$ will thus be set to the identity matrix in what follows.

\subsection{Block sparse PCA via cardinality penalty}

The single-unit cardinality-penalized case can also be naturally extended
to the block case:
\begin{equation}
\blockc(\sparsity)\eqdef \max_{\substack{X \in
\St{m}{p}\\Z \in \sphereprod{n}{m}}} \trace(\Diag(X^T A Z N)^2) - \sparsity \card{Z}
\label{eq:block_spca3},
\end{equation}
where $\sparsity \geq 0$ is the sparsity inducing parameter and $N=\Diag(\mu_1, \ldots, \mu_m)$ with positive entries on the diagonal.
In the case $\sparsity=0$, problem (\ref{eq:block_spca4}) is equivalent to (\ref{eq:brockett}) and therefore corresponds to PCA, provided that all $\mu_i$ are distinct.

\bigskip
\textbf{Reformulation.} Again, this block formulation is completely decoupled in the columns of $Z$,
\begin{equation*}
\blockc(\sparsity) = \underset{X \in \St{m}{p}}{\max}
\sum_{j=1}^m  \underset{z_j \in \sphere{n}}{\max} (\mu_j x_j^T A
z_j)^2 -\sparsity \card{z_{j}},
\end{equation*}
so that the solution (\ref{eq:spca_single_l0_2}) of the single unit case provides the optimal columns $z_i$:
\begin{equation}
\label{eq:block_spca_card_solution}
z_{ij}^* = z_{ij}^*(\sparsity) = \frac{[\signum((\mu_j a_i^T x_j)^2-\sparsity)]_+ \mu_j a_i^T x_j}{\sqrt{\sum_{k=1}^n [\signum((\mu_j a_k^T x_j)^2-\sparsity)]_+  \mu_j^2 (a_k^T x_j)^2}}.
\end{equation}
The reformulation of problem (\ref{eq:block_spca3}) is thus
\begin{equation}
\label{eq:block_spca4}\boxed{\blockc(\sparsity)= \max_{X \in \St{m}{p}} \sum_{j=1}^m \sum_{i=1}^n
[(\mu_j a_i^T x_j)^2 - \sparsity]_+,}
\end{equation}
which maximizes a convex function $f: \R^{p \times m} \rightarrow \R$ on the Stiefel manifold $\St{m}{p}$.

\bigskip
\textbf{Sparsity.} For a solution $X^*$ of (\ref{eq:block_spca4}), the active entries $z_{ij}^*$ of $Z^*$ are given by the condition \[(\mu_j a_i^T x_j^*)^2 > \sparsity.\] Hence for
$\sparsity > \underset{i,j}{\max} \; \mu_j \|a_i\|_2^2,$ the optimal solution of (\ref{eq:block_spca3}) is $Z^*=0$.

\bigskip
\section{A gradient method for maximizing convex functions} \label{sec:theory}

By $\E$ we denote an arbitrary finite-dimensional vector space;
$\E^*$ is its conjugate, i.e. the space of all linear functionals on
$\E$. We equip these spaces with norms given by (\ref{eq:norms}).

In this section we propose and analyze a simple gradient-type method
for maximizing a convex function $f:\E\to \R$ on a compact set $\Q$:
\begin{equation}\tag{P}
\boxed{f^* = \max_{x\in \Q} f(x).}
\end{equation}

Unless explicitly stated otherwise, we will \emph{not} assume $f$ to
be differentiable. By $f'(x)$ we denote any subgradient of function
$f$ at $x$. By $\partial f(x)$ we denote its subdifferential.

At any point $x \in \Q$ we introduce some measure for the
first-order optimality conditions:
$$
\Delta(x) \eqdef \max\limits_{y \in \Q} \ve{f'(x)}{y-x}.
$$
Clearly, $\Delta(x) \geq 0$ and it vanishes only at the points where
the gradient $f'(x)$ belongs to the normal cone to the set ${\rm
Conv}(\Q)$ at $x$.\footnote{ The normal cone to the set
$\conv(\Q)$ at $x \in \Q$ is \emph{smaller} than the normal cone to
the set $\Q$. Therefore, the optimality condition $\Delta(x)=0$ is
\emph{stronger} than the standard one.}

We will use the following notation:
\begin{equation}
y(x) \indef \Arg \max_{y\in \Q} \ve{f'(x)}{y-x}.
\end{equation}

\subsection{Algorithm}
Consider the following simple algorithmic scheme.

\begin{algorithm}[h!]
\dontprintsemicolon \SetKwInOut{Input}{input}
\SetKwInOut{Output}{output}
 \Input{Initial iterate $x_0 \in \E$.}
 \Output{$x_k$, approximate solution of (P)}
 \Begin{ $k \longleftarrow 0$\\
 \Repeat{a stopping criterion is satisfied}
 {$x_{k+1} \in {\rm Arg}\max \{ f(x_k) +
\ve{f'(x_k)}{y-x_k} \suchthat y\in \Q\}$\\$k \longleftarrow k+1$}}
 \caption{Gradient scheme \label{algo1}}
\end{algorithm}

Note that for example in the special case $\Q=r\cdot \sph \eqdef r\cdot \{ x \in \E
\suchthat \| x \| = r\}$ or \\$\Q= r \cdot \ball\eqdef r \cdot \{x\in \E \suchthat \|x\|\leq r\},$ the main step of Algorithm~\ref{algo1} can
be written in an explicit form:
\begin{equation}\label{eq:main_step:ball_setting}
y(x_k) = x_{k+1} = r\frac{\G^{-1}f'(x_k)}{\|f'(x_k)\|_*}.
\end{equation}

\subsection{Analysis}

Our first convergence result is straightforward. Denote $\Delta_k \eqdef
\min\limits_{0 \leq i \leq k} \Delta(x_i)$.
\begin{thrm}\label{thm:main1}
Let sequence $\{ x_k \}_{k=0}^{\infty}$ be generated by
Algorithm~\ref{algo1} as applied to a convex function~$f$. Then the
sequence $\{ f(x_k) \}_{k=0}^{\infty}$ is monotonically increasing
and $\lim\limits_{k \to \infty} \Delta(x_k) = 0$. Moreover,
\begin{equation}\label{eq:thm_main1:Rate}
\Delta_k \leq \frac{f^* - f(x_0)}{k+1}.
\end{equation}
\end{thrm}
\begin{proof}
From convexity of $f$ we immediately get
\begin{equation*}
f(x_{k+1})  \geq f(x_k) + \ve{f'(x_k)}{x_{k+1}-x_k} = f(x_k) +
\Delta(x_k),
\end{equation*}
and therefore, $f(x_{k+1})\geq f(x_k)$ for all $k$. By summing up
these inequalities for $k=0,1,\dots,N-1$, we obtain
\[
f^* - f(x_0) \geq f(x_k) - f(x_0) \geq \sum\limits_{i=0}^k
\Delta(x_i),
\]
and the result follows.
\end{proof}

For a sharper analysis, we need some technical assumptions on $f$
and $\Q$.

\begin{ssmptn}\label{ass:1}
The norms of the subgradients of $f$ are bounded from below on $\Q$
by a positive constant, i.e.
\begin{equation}\label{eq:delta_f}
\delta_f \eqdef \min_{\substack{x\in \Q\\
f'(x)\in \partial f(x)}} \|f'(x)\|_* >0.
\end{equation}
\end{ssmptn}
This assumption is not too binding because of the following result.
\begin{prpstn}\label{prop:Ass2}
Assume that there exists a point $\bar x \not\in \Q$ such that $f(\bar x) <
f(x)$ for all $x \in \Q$. Then
\[
\delta_f \geq \left[ \min\limits_{x \in \Q} f(x) - f(\bar x) \right]
/ \left[ \max\limits_{x \in \Q} \| x - \bar x \| \right]
> 0.
\]
\end{prpstn}
\begin{proof}
Because $f$ is convex, for any $x \in \Q$ we have
\[
0 < f(x) - f(\bar x) \leq \ve{f'(x)}{x - \bar x} \leq \| f'(x) \|_*
\cdot \| x - \bar x \|.
\]
\end{proof}

For our next convergence result we need to assume either strong
convexity of $f$ or strong convexity of the set $\conv(\Q)$.

\begin{ssmptn}\label{ass:2}
Function $f$ is \emph{strongly convex}, i.e. there exists a constant
$\sigma_f>0$ such that for any $x,y\in \E$
\begin{equation}\label{eq:convex}
f(y)\geq f(x) + \ve{f'(x)}{y-x} + \frac{\sigma_f}{2} \| y - x \|^2.
\end{equation}
\end{ssmptn}

Convex functions satisfy this inequality for \emph{convexity
parameter} $\sigma_f=0$.

\begin{ssmptn}\label{ass:3}
The set $\conv (\Q)$ is \emph{strongly convex}. This means that
there exists a constant $\sigma_{\Q}>0$ such that for any $x, y \in
\conv(\Q)$ and $\alpha \in [0,1]$ the following inclusion holds:
\begin{equation}\label{eq:set_strongly_convex}
\alpha x + (1 - \alpha) y + \frac{\sigma_{\Q}}{2}  \alpha(1-\alpha)\|
x - y \|^2 \cdot \sph\subset \conv( \Q ).
\end{equation}
\end{ssmptn}

Convex sets satisfy this inclusion for \emph{convexity parameter}
$\sigma_{\Q}=0$. It can be shown (see Appendix), that level sets of strongly
convex functions with Lipschitz continuous gradient are again strongly convex.
An example of such a function is the simple quadratic
$x \mapsto \|x\|^2$. The level sets of this function correspond to Euclidean balls of
varying sizes.

As we will see in Theorem~\ref{thm:main}, a better analysis of Algorithm~\ref{algo1} is possible if $\conv(\Q)$, the convex hull
of the feasible set of problem (P), is strongly convex. Note that in the case of the two formulations (\ref{eq:SPCA3}) and (\ref{eq:SPCA4}) of the
sparse PCA problem, the feasible set $\Q$ is the unit Euclidean sphere. Since the convex hull of the unit sphere is the unit ball, which is a strongly
convex set, the feasible set of our sparse PCA formulations satisfies Assumption~\ref{ass:3}.

In the special case $\Q=r \cdot
\sph$ for some $r > 0$, there is a simple proof that Assumption~\ref{ass:3} holds with
$\sigma_{\Q}=\frac{1}{r}$. Indeed, for any $x, y \in \E$ and $\alpha
\in [0,1]$, we have
\[
\begin{array}{rcl}
\| \alpha x + (1 - \alpha) y \|^2 & = & \alpha^2 \| x \|^2 + (1 -
\alpha)^2 \| y \|^2 + 2 \alpha (1 - \alpha)
\ve{Gx}{y}\\
\\
& = & \alpha \| x \|^2 + (1 - \alpha) \| y \|^2 - \alpha(1 - \alpha)
\| x - y \|^2.
\end{array}
\]
Thus, for $x, y \in r\cdot \sph$ we obtain:
\[
\| \alpha x + (1 - \alpha) y \| = \left[ r^2 - \alpha(1 - \alpha) \|
x - y \|^2 \right]^{1/2} \leq r  - \frac{1}{2r} \alpha(1 - \alpha)
\| x - y \|^2.
\]
Hence, we can take $\sigma_{\Q} = \frac{1}{r}$.

\bigskip
The relevance of Assumption \ref{ass:3} is
justified by the following technical observation.
\begin{prpstn}\label{prop:Ass3}
Let Assumption \ref{ass:3} be satisfied. Then for any $x \in Q$ the following holds:
\begin{equation}\label{eq:delta_from_strong_conv_of_Q}
\Delta(x) \geq \frac{\sigma_{\Q}}{2}  \| f'(x) \|_* \cdot \| y(x) -
x \|^2.
\end{equation}
\end{prpstn}
\begin{proof}
Fix an arbitrary $x \in \Q$. Note that
\[
\ve{f'(x)}{y(x) - y} \geq 0, \quad y \in \conv(\Q).
\]
We will use this inequality for
\[
y = y_{\alpha} \eqdef x + \alpha (y(x) - x) + \frac{\sigma_{\Q}}{2}  \alpha (1
- \alpha)  \| y(x) - x \|^2 \cdot \frac{G^{-1} f'(x)}{\|
f'(x) \|_*}, \; \alpha \in [0,1].
\]
In view of Assumption \ref{ass:3}, $y_{\alpha} \in \conv(\Q)$.
Therefore,
\[
0 \geq \ve{f'(x)}{y_{\alpha} - y(x)} = (1-\alpha) \ve{f'(x)}{x
-y(x)}+ \frac{\sigma_{\Q}}{2}  \alpha (1 - \alpha)  \| y(x) - x \|^2
\cdot \| f'(x) \|_*.
\]
Since $\alpha$ is an arbitrary value from $[0,1]$, the result follows.
\end{proof}

\bigskip
We are now ready to refine our analysis of Algorithm~1.

\begin{thrm}[Convergence]\label{thm:main}
Let $f$ be convex and let Assumption \ref{ass:1} and at least one of
Assumptions \ref{ass:2} and \ref{ass:3} be satisfied. If $\{x_k\}$
is the sequence of points generated by Algorithm 1, then
\begin{equation}\label{eq:thm1proof}
\sum\limits_{k=0}^N \|x_{k+1}-x_k\|^2 \leq  \frac{2(f^* -
f(x_0))}{\sigma_{\Q} \delta_f + \sigma_f}.
\end{equation}
\end{thrm}
\begin{proof}
Indeed, in view of our assumptions and Proposition \ref{prop:Ass3},
we have
\[
f(x_{k+1}) - f(x_k) \geq  \Delta(x_k) + \frac{\sigma_f}{2}
\|x_{k+1}-x_k\|^2\geq \frac{1}{2} (\sigma_{\Q} \delta_f + \sigma_f)
\|x_{k+1}-x_k\|^2.
\]
\end{proof}

We cannot in general guarantee that the algorithm will converge to a
unique local maximizer. In particular, if started from a local
minimizer, the method will not move away from this point. However,
the above statement guarantees that the set of its limit points is
connected and all of them satisfy the first-order optimality
condition.

\subsection{Maximization with spherical constraints}
\label{sec:vector_setting}

Consider $\E=\E^*=\R^p$ with $\G=I_p$ and $\ve{s}{x}=\sum_i s_i x_i$,
and let \[\Q=r\cdot \sphere{p} = \{x\in \R^p \suchthat \|x\| = r\}.\] Problem (P)
takes on the form:
\[
\boxed{f^* = \max_{\substack{x\in r \cdot \sphere{p}}} f(x).}
\]
Since $\Q$ is strongly convex ($\sigma_\Q=\tfrac{1}{r}$),  Theorem
\ref{thm:main} is meaningful for any convex function $f$
($\sigma_f\geq 0$). We have already noted (see (\ref{eq:main_step:ball_setting})) that the main step of Algorithm~\ref{algo1} can
be written down explicitly. Note that the single-unit sparse PCA formulations (\ref{eq:SPCA3}) and (\ref{eq:SPCA4}) conform to this setting. The following examples illustrate the connection to classical algorithms.

\begin{xmpl}[Power method] In the special case of a quadratic objective function $f(x)~=~\tfrac{1}{2}x^TCx$ for some $C \in \Scone^p_{++}$ on the unit sphere ($r=1$), we have
\[f^*=\tfrac{1}{2}\lambda_{\text{max}}(C),\] and Algorithm~\ref{algo1} is equivalent to the
\emph{power iteration method} for computing the largest eigenvalue of $C$ (\citet{Golub96}).
Hence for $\Q=\sphere{p}$, we can think of our scheme as a
generalization of the power method. Indeed, our algorithm performs the following iteration:
\[x_{k+1} = \frac{Cx_k}{\|Cx_k\|}, \quad k \geq 0.\]
Note that both $\delta_{f}$ and  $\sigma_{f}$  are equal to the smallest eigenvalue of $C$, and hence the right-hand side of (\ref{eq:thm1proof}) is equal to
\begin{equation}\label{eq:case1:estimate}\frac{\lambda_{\text{max}}(C) - x_0^T C x_0}{2\lambda_{\text{min}}(C)}.\end{equation}
\end{xmpl}

\begin{xmpl}[Shifted power method] If $C$ is not positive semidefinite in the previous example, the objective function is not
convex and our results are not applicable. However, this complication can be circumvented by instead running the algorithm with the shifted quadratic function
\[\hat{f}(x) = \frac{1}{2}x^T (C + \level I_p) x,\]
where $\level>0$ satisfies $\hat{C} = \level I_p + C \in \Scone^p_{++}$. On the feasible set, this change only adds a constant term to the objective function. The method, however, produces different sequence of iterates. Note that the constants $\delta_f$ and $\sigma_f$ are also affected and, correspondingly, the estimate (\ref{eq:case1:estimate}).
\end{xmpl}

\subsection{Maximization with orthonormality constraints} \label{sec:matrix_setting}

Consider $\E=\E^*=\R^{p\times m}$, the space of $p\times m$ real matrices, with $m\leq p$. Note that for $m=1$ we recover the setting of the previous section. We assume this space is equipped with the trace inner product: $\ve{X}{Y}=\trace(X^TY)$. The induced norm, denoted by $\|X\|_F\eqdef\ve{X}{X}^{1/2}$, is the Frobenius norm (we let $\G$ be the identity operator).
We can now consider various feasible sets, the simplest being a ball or a sphere. Due to nature of applications in this paper, let us concentrate on the situation when $\Q$ is a special subset of the sphere with radius $r= \sqrt{m}$, the Stiefel manifold $\St{m}{p}$:
\[
\Q=\St{m}{p} = \{X \in \R^{p\times m} \suchthat X^TX = I_m\}.
\]
Problem (P) then takes on the following form:
\[
\boxed{f^* = \max_{X\in \St{m}{p}} f(X).}
\]
Note that $\conv(\Q)$ is not strongly convex ($\sigma_\Q = 0$), and
hence Theorem \ref{thm:main} is meaningful only if $f$ is strongly
convex ($\sigma_f > 0$). At every iteration, the algorithm needs to maximize a linear function over the Stiefel manifold. The following standard result shows how this can be done.

\begin{prpstn}
\label{thm:iteration_on_stiefel}
Let $C \in \R^{p \times m}$, with $m\leq p$, and denote by $\sigma_i(C)$, $i=1,\dots,m$, the singular values of $C$. Then
\begin{equation}
\label{eq:iteration_stiefel} \max_{X\in \St{m}{p}} \ve{C}{X}  = \trace[(C^T C)^{1/2}] = \sum_{i=1}^m \sigma_i(C),
\end{equation}
and a maximizer $X^*$ is given by the $U$ factor in the polar decomposition of $C$:
\[C = U P, \quad U \in \St{m}{p}, \; P\in \Scone_+^m.\]
If $C$ is of full rank, then we can take $X^* = C (C^TC)^{-1/2}$.
\end{prpstn}
\begin{proof} Existence of the polar factorization in the nonsquare case is covered by Theorem 7.3.2 in \citet{HJ85}.
Let $C=V\Sigma W^T$ be the singular value decomposition of $A$; that is, $V$ is $p\times p$ orthonormal, $W$ is $m\times m$ orthonormal, and $\Sigma$ is $p\times m$ diagonal with values $\sigma_i(A)$ on the diagonal. Then
\begin{align*}\max_{X \in \St{m}{p}} \ve{C}{X} &= \max_{X \in \St{m}{p}} \ve{V\Sigma W^T}{X}\\
&= \max_{X \in \St{m}{p}} \trace \Sigma (W^TX^T V)\\
&= \max_{Z \in \St{m}{p}} \trace \Sigma Z^T= \max_{Z \in \St{m}{p}} \sum_{i=1}^m \sigma_i(C) z_{ii} \leq \sum_{i}^m \sigma_i(C).
\end{align*}

The third equality follows since the function $X \mapsto V^T X W$ maps $\St{m}{p}$ onto itself. It remains to note that \begin{align*}\ve{C}{U} = \trace P = \sum_{i} \lambda_i(P) = \sum_{i} \sigma_i(P) = \trace (P^T P)^{1/2} = \trace (C^T C)^{1/2} = \sum_i \sigma_i(C),\end{align*}
Finally, in the full rank case we have $\ve{C}{X^*} = \trace C^T C (C^TC)^{-1/2} = \trace (C^T C)^{1/2}$.

\end{proof}

In the sequel, the symbol $\uf(C)$ will be used to denote the $U$ factor of the polar decomposition of matrix $C \in \R^{p \times m}$, or equivalently, $\uf(C)=C(C^TC)^{-1/2}$ if $C$ is of full rank. In view of the above result, the main step of Algorithm~\ref{algo1} can be written in the form
\begin{equation}\label{eq:main_step:matrix_setting}x_{k+1}=\uf(f'(x_k)).\end{equation}

Note that the block sparse PCA formulations (\ref{eq:block_spca2}) and (\ref{eq:block_spca4}) conform to this setting. Here is one more example:

\begin{xmpl}[Rectangular Procrustes Problem] Let $C,X\in \R^{p\times m}$ and $D\in \R^{p\times p}$ and consider the following problem:
\begin{equation}\label{eq:case2:Frob} \min \{\|C-DX\|^2_F \suchthat X^T X=I_m\}.\end{equation}
Since $\|C-DX\|_F^2 = \|C\|_F^2 + \ve{DX}{DX} - 2\ve{CD}{X}$, by a similar shifting technique as in the previous example we can cast problem (\ref{eq:case2:Frob}) in the following form
\[\max \{\level\|X\|_F^2 - \ve{DX}{DX}+2\ve{CD}{X} \suchthat X^TX = I_m\}.\]
For $\level>0$ large enough, the new objective function will be strongly convex. In this case our algorithm becomes similar to the gradient
method proposed by \citet{FNVD08}.

The standard Procrustes problem in the literature is a special case of (\ref{eq:case2:Frob}) with $p=m$.
\end{xmpl}

\bigskip
\section{Algorithms for sparse PCA}\label{sec:sparse_pca_algorithms}

The application of our general method (Algorithm~\ref{algo1}) to the four sparse
PCA formulations of Section \ref{sec:spca}, i.e., (\ref{eq:SPCA3}), (\ref{eq:SPCA4}), (\ref{eq:block_spca2}) and (\ref{eq:block_spca4}), leads to Algorithms \ref{algo2}, \ref{algo3}, \ref{algo4} and \ref{algo5} below, that provide a locally optimal pattern of sparsity for a matrix $Z \in \sphereprod{n}{m}$.\footnote{This section discusses the general block sparse PCA problem. The single-unit case corresponds to the particular case $m=1$.} This pattern is defined as a matrix $P\in \R^{n\times m}$ such that $p_{ij} = 0$ if the loading $z_{ij}$ is active and $p_{ij} = 1$ otherwise. So $P$ is an indicator of the coefficients of $Z$ that are zeroed by our method. The computational complexity of the single-unit algorithms (Algorithms \ref{algo2} and \ref{algo3}) is $\mathcal{O}(n p)$ operations per iteration. The block algorithms (Algorithms \ref{algo4} and \ref{algo5}) have complexity $\mathcal{O}(n p m)$ per iteration.

\subsection{Methods for pattern-finding}

\begin{algorithm}[h]
\dontprintsemicolon \SetKwInOut{Input}{input}
\SetKwInOut{Output}{output}
 \Input{Data matrix $A \in \R^{p \times n}$\\ Sparsity-controlling parameter
$\sparsity \geq 0$\\ Initial iterate $x \in \sphere{p}$}
 \Output{A locally optimal sparsity pattern $P$}
 \Begin{  \Repeat{a stopping criterion is satisfied}{$x \longleftarrow
\sum_{i=1}^n [|a_i^T x| -
 \sparsity]_+ \signum(a_i^T x) a_i$\\
 $x \longleftarrow \frac{x}{\|x\|}$}
Construct vector $P \in \R^{n}$ such that
$\left\{\begin{array}{ll}
    p_{i}=0 & \text{ if }|a_i^T x| > \sparsity\\
    p_{i}=1 & \text{ otherwise.}
\end{array}\right.$}
 \caption{Single-unit sparse PCA method based on the $\ell_1$-penalty
(\ref{eq:SPCA3}) \label{algo2}}
\end{algorithm}

\begin{algorithm}[h]
\dontprintsemicolon \SetKwInOut{Input}{input}
\SetKwInOut{Output}{output}
 \Input{Data matrix $A \in \R^{p \times n}$\\ Sparsity-controlling parameter
$\sparsity \geq 0$\\ Initial iterate $x \in \sphere{p}$}
 \Output{A locally optimal sparsity pattern $P$}
 \Begin{ \Repeat{a stopping criterion is satisfied}{$x \longleftarrow
\sum_{i=1}^n [\signum((a_i^T x)^2 -
 \sparsity)]_+ \; a_i^T x \; a_i$\\
 $x \longleftarrow \frac{x}{\|x\|}$}
Construct vector $P \in \R^{n}$ such that
$\left\{\begin{array}{ll}
    p_{i}=0 & \text{ if } (a_i^T x)^2 > \sparsity\\
    p_{i}=1 & \text{ otherwise.}
\end{array}\right.$}
 \caption{Single-unit sparse PCA algorithm based on the $\ell_0$-penalty
(\ref{eq:SPCA4}) \label{algo3}}
\end{algorithm}

\begin{algorithm}[h]
\dontprintsemicolon \SetKwInOut{Input}{input}
\SetKwInOut{Output}{output}
 \Input{Data matrix $A \in \R^{p \times n}$\\
 Sparsity-controlling parameter $\sparsity \geq 0$\\
Initial iterate $X \in \St{m}{p}$}
 \Output{A locally optimal sparsity pattern $P$}
 \Begin{ \Repeat{a stopping criterion is satisfied}{
\For{$j= 1, \ldots, m$}{$x_j \longleftarrow
\sum_{i=1}^n  [|a_i^T x_j| -
 \sparsity]_+ \signum(a_i^T x) a_i$\\}
 $X \longleftarrow \uf(X)$}
Construct matrix $P \in \R^{n \times m}$ such that
$\left\{\begin{array}{ll}
    p_{ij}=0 & \text{ if }|a_i^T x_j| > \sparsity\\
    p_{ij}=1 & \text{ otherwise.}
\end{array}\right.$}
 \caption{Block Sparse PCA algorithm based on the $\ell_1$-penalty
(\ref{eq:block_spca2}) \label{algo4}}
\end{algorithm}

\begin{algorithm}[h]
\dontprintsemicolon \SetKwInOut{Input}{input}
\SetKwInOut{Output}{output}
 \Input{Data matrix $A \in \R^{p \times n}$\\
Sparsity-controlling parameter $\sparsity \geq 0$\\
Initial iterate $X \in \St{m}{p}$}
 \Output{A locally optimal sparsity pattern $P$}
 \Begin{ \Repeat{a stopping criterion is satisfied}{
\For{$j= 1, \ldots, m$}{$x_j \longleftarrow
\sum_{i=1}^n [\signum((a_i^T x_j)^2 -
 \sparsity)]_+\;  a_i^T x_j \; a_i$\\}
 $X \longleftarrow \uf(X)$}
Construct matrix $P \in \R^{n \times m}$ such that
$\left\{\begin{array}{ll}
    p_{ij}=0 & \text{ if }(a_i^T x_j)^2 > \sparsity\\
    p_{ij}=1 & \text{ otherwise.}
\end{array}\right.$}
 \caption{Block Sparse PCA algorithm based on the $\ell_0$-penalty
(\ref{eq:block_spca4}) \label{algo5}}
\end{algorithm}

\subsection{Post-processing}

Once a ``good'' sparsity pattern $P$ has been identified, the active entries of $Z$ still have to be filled. To this end, we consider the optimization problem,
\begin{equation}
\label{eq:opt_imposed_sparse_pattern}
(X^*,Z^*) \eqdef
\arg\max_{\substack{X \in \St{m}{p}\\ Z \in \sphereprod{n}{m} \\ Z_P=0}}
\trace(X^T A Z N),
\end{equation}
where $Z_P$ denotes the entries of $Z$ that are constrained to zero and $N=\Diag(\mu_1,\ldots,\mu_m)$ with strictly positive $\mu_i$. Problem (\ref{eq:opt_imposed_sparse_pattern}) assigns the active part of the loading vectors $Z$ to maximize the variance explained by the resulting components. By $Z_{\bar{P}}$, we refer to the complement of $Z_P$, i.e., to the active entries of $Z$. In the single-unit case $m=1$, an explicit solution of (\ref{eq:opt_imposed_sparse_pattern}) is available,
\begin{equation}
\label{eq:solution_postprocessing_single_unit}
\begin{array}{ll}
 X^*&=u, \\
 Z^*_{\bar{P}}&= v \; \text { and } \; Z^*_{P}=0,
\end{array}
\end{equation}
where $\sigma u v^T$ with $\sigma >0$, $u \in \ball^p$ and $v \in
\ball^{\card{\bar{P}}}$ is a rank one singular value decomposition
of the matrix $A_{\bar{P}}$, that corresponds to the submatrix of
$A$ containing the columns related to the active entries.

Although an exact solution of (\ref{eq:opt_imposed_sparse_pattern}) is hard to compute in the block case $m>1$, a local maximizer can be efficiently computed by optimizing alternatively with respect to one variable while keeping the other ones fixed. The following lemmas provide an explicit solution to each of these subproblems.
\begin{lmm} For a fixed $Z \in \sphereprod{n}{m}$, a solution $X^*$ of
\[\max_{X \in \St{m}{p}}\trace(X^T A Z N)\]
is provided by the $U$ factor of the polar decomposition of the product $A Z N$.
\end{lmm}
\begin{proof}
See Proposition \ref{thm:iteration_on_stiefel}.
\end{proof}

\begin{lmm}
The solution
\begin{equation}
\label{eq:opt_prob_Z}
Z^* \eqdef \arg \max_{\substack{Z \in \sphereprod{n}{m} \\ Z_P =0}} \trace(X^T A Z N),
\end{equation}
is at any point $X \in \St{m}{p}$ defined by the two conditions $Z^*_{\bar{P}}=(A^T X N D)_{\bar{P}}$ and $Z^*_P=0$, where $D$ is a positive diagonal matrix that normalizes each column of $Z^*$ to unit norm, i.e., \[D~=~\mathrm{Diag}(N X^T A A^T X N)^{-\frac{1}{2}}.\]
\end{lmm}
\begin{proof} The Lagrangian of the optimization problem (\ref{eq:opt_prob_Z}) is
\[\mathcal{L}(Z,\Lambda_1,\Lambda_2)= \trace(X^T A Z N) - \trace(\Lambda_1 (Z^T Z-I_m)) - \trace(\Lambda_2^T Z),
\]
where the Lagrangian multipliers $\Lambda_1\in \R^{m \times m}$ and $\Lambda_2 \in \R^{n \times m}$ have the following properties: $\Lambda_1$ is an invertible diagonal matrix and $(\Lambda_{2})_{\bar{P}}=0$.
The first order optimality conditions of (\ref{eq:opt_prob_Z}) are thus
\begin{align*}
A^T X N - 2 Z \Lambda_1 - \Lambda_2&=0\\
\mathrm{Diag}(Z^T Z)&=I_m\\
Z_P& =0.
\end{align*}
Hence, any stationary point $Z^*$ of (\ref{eq:opt_prob_Z}) satisfies $Z_{\bar{P}}^*=(A^T X N D)_{\bar{P}}$ and $Z^*_P=0$, where $D$ is a diagonal matrix that normalizes the columns of $Z^*$ to unit norm. The second order optimality condition imposes the diagonal matrix $D$ to be positive. Such a $D$ is unique and given by $D=\mathrm{Diag}(N X^T A A^T X N)^{-\frac{1}{2}}.$
\end{proof}\\

The alternating optimization scheme is summarized in Algorithm \ref{algo:alternate_optim}, which computes a local solution of (\ref{eq:opt_imposed_sparse_pattern}).
\begin{algorithm}[h]
\dontprintsemicolon \SetKwInOut{Input}{input}
\SetKwInOut{Output}{output}
 \Input{Data matrix $A \in \R^{p \times n}$\\
 Sparsity pattern $P \in \R^{n \times m}$\\
Matrix $N=\Diag(\mu_1,\ldots, \mu_m)$\\
Initial iterate $X \in \St{m}{p}$}
 \Output{A local minimizer $(X, Z)$ of (\ref{eq:opt_imposed_sparse_pattern})}
 \Begin{\Repeat{a stopping criterion is satisfied}{
$Z \longleftarrow A^T X N$ \\
$Z \longleftarrow Z \; \Diag(Z^T Z)^{-\frac{1}{2}}$\\
$Z_P \longleftarrow 0$\\
$X \longleftarrow \uf(A Z N)$}}
 \caption{Alternating optimization scheme for solving (\ref{eq:opt_imposed_sparse_pattern}) \label{algo:alternate_optim}}
\end{algorithm}
It should be noted that Algorithm \ref{algo:alternate_optim} is a postprocessing heuristic that, strictly speaking, is required only for the $\ell_1$ block formulation (Algorithm \ref{algo4}). In fact, since the cardinality penalty only depends on the sparsity pattern $P$ and not on the actual values assigned to $Z_{\bar{P}}$, a solution $(X^*,Z^*)$ of Algorithms \ref{algo3} or \ref{algo5} is also a local maximizer of (\ref{eq:opt_imposed_sparse_pattern}) for the resulting pattern $P$. This explicit solution provides a good alternative to Algorithm \ref{algo:alternate_optim}. In the single unit case with $\ell_1$ penalty (Algorithm \ref{algo2}), the solution (\ref{eq:solution_postprocessing_single_unit}) is available.

\subsection{Sparse PCA algorithms}

To sum up, in this paper we propose four sparse PCA algorithms, each combining a method to identify a ``good'' sparsity pattern with a method to fill the active entries of the $m$ loading vectors. They are summarized in Table \ref{tbl:new_methods}.\footnote{Our algorithms are named $\mathsf{GPower}$ where the ``G'' stands for  \emph{generalized} or \emph{gradient}.}
\begin{table}[h!] \tabsize \centerline{
\begin{tabular}{l l l}
\hline
 & Computation of $P$ & Computation of $Z_{\bar{P}}$\\
\hline
$\Asingle$ & Algorithm \ref{algo2} & Equation (\ref{eq:solution_postprocessing_single_unit})\\
$\Asinglec$& Algorithm \ref{algo3} & Equation (\ref{eq:spca_single_l0_2}) \\
$\Ablock$ & Algorithm \ref{algo4} &  Algorithm \ref{algo:alternate_optim}\\
$\Ablockc$ & Algorithm \ref{algo5} & Equation (\ref{eq:block_spca_card_solution})\\
 \hline
\end{tabular}}
 \caption{\capsize New algorithms for sparse PCA.} \label{tbl:new_methods}
\end{table}

\subsection{Deflation scheme.}

For the sake of completeness, we recall a classical deflation process for
computing $m$ sparse principal components with a single-unit algorithm (\combib{d'Aspremont et al. }\citet{Aspremont07}). Let $z \in \R^n$ be a unit-norm sparse loading vector of the data $A$. Subsequent directions can be sequentially obtained by computing a dominant sparse component of the residual matrix
$A-x z^T$, where $x = A z$ is the vector that solves
\[\underset{x \in \R^p}{\min} \| A-x z^T\|_F.\]

\bigskip
\section{Numerical experiments}
\label{sec:num_sim}
In this section, we evaluate
the proposed power algorithms against existing sparse PCA methods. Three competing methods are considered in this study: a greedy scheme aimed at computing a local maximizer of (\ref{eq:RQ_l0}) (\combib{d'Aspremont et al. }\citet{Aspremont07b}), the $\SPCA$ algorithm (\combib{Zou et al. }\citet{Zou04}) and the sPCA-rSVD algorithm (\combib{Shen and Huang }\citet{shen08}). We do not include the $\DSPCA$ algorithm (\combib{d'Aspremont et al. }\citet{Aspremont07}) in our numerical study. This method solves a convex relaxation of the sparse PCA problem and has a large computational complexity of $\mathcal{O}(n^3)$ compared to the other methods. Table \ref{tbl:methods} lists the considered algorithms.
\begin{table}[H] \tabsize \centerline{
\begin{tabular}{l l}
\hline
$\Asingle$ & Single-unit sparse PCA via $\ell_1$-penalty\\
$\Asinglec$& Single-unit sparse PCA via $\ell_0$-penalty\\
$\Ablock$ & Block sparse PCA via $\ell_1$-penalty\\
$\Ablockc$ & Block sparse PCA via $\ell_0$-penalty\\
$\Greedy$ & Greedy method\\
$\SPCA$  & SPCA algorithm\\
$\rSVD$  & sPCA-rSVD algorithm with an $\ell_1$-penalty (``soft thresholding'')\\
$\rSVDc$ & sPCA-rSVD algorithm with an $\ell_0$-penalty (``hard thresholding'')\\
 \hline
\end{tabular}}
 \caption{\capsize Sparse PCA algorithms we compare in this section.} \label{tbl:methods}
\end{table}

These algorithms are compared on random data (Section \ref{sec:num_exp:examples}) as well as on real data (Section \ref{sec:num_exp:genes}). All numerical experiments are performed in \textsf{MATLAB}. Our implementations of the $\mathsf{GPower}$ algorithms are initialized at a point for which the associated sparsity pattern has \emph{at least one} active element. In case of the single-unit algorithms, such an initial iterate $x \in \sphere{p}$ is chosen parallel to the column of $A$ with the largest norm, i.e.,
\begin{equation} \label{eq:initial_point} x = \frac{a_{i^*}}{\|a_{i^*}\|_2}, \quad \text{ where } \quad i^* = \arg \max_i \|a_i\|_2.\end{equation}
For the block $\mathsf{GPower}$ algorithms, a suitable initial iterate $X \in \St{m}{p}$ is constructed in a block-wise manner as $X =[ x | X_{\perp}]$, where $x$ is the unit-norm vector (\ref{eq:initial_point}) and $X_{\perp} \in \St{m-1}{p}$ is orthogonal to $x$, i.e., $x^T X_{\perp} = 0$. We stop the $\mathsf{GPower}$ algorithms once the relative change of the objective function is small: \[\frac{f(x_{k+1})-f(x_k)}{f(x_k)} \leq \epsilon = 10^{-4}.\]
\textsf{MATLAB} implementations of the $\SPCA$ algorithm and the greedy algorithm have been rendered available by \combib{Zou et al. }\citet{Zou04} and \combib{d'Aspremont et al.}\citet{Aspremont07b}. We have, however, implemented the sPCA-rSVD algorithm on our own (Algorithm 1 in \combib{Shen and Huang }\citet{shen08}), and use it with the same stopping criterion as for the $\mathsf{GPower}$ algorithms. This algorithm initializes with the best rank-one approximation of the data matrix. This is done with the \verb"svds" function in \textsf{MATLAB}.

Given a data matrix $A \in \R^{p \times n}$, the considered sparse PCA algorithms provide $m$ unit-norm sparse loading vectors stored in the matrix $Z\in \sphereprod{n}{m}$. The samples of the associated components are provided by the $m$ columns of the product $AZ$. The variance explained by these $m$ components is an important comparison criterion of the algorithms. In the simple case $m=1$, the variance explained by the component $A z$ is \[\Var(z)=z^T A^T A z.\]
When $z$ corresponds to the first principal loading vector, the variance is $\Var(z) = \sigma_{\text{max}}(A)^2$. In the case $m>1$, the derived components are likely to be correlated. Hence, summing up the variance explained individually by each
of the components overestimates the variance explained
simultaneously by all the components. This motivates the notion of \emph{adjusted
variance}
proposed by \combib{Zou et al. }\citet{Zou04}. The adjusted variance of the $m$ components $Y = A Z$ is defined as
\[ \AVar\; Z = \trace R^2,\]
where $Y= Q R$ is the QR decomposition of the components sample matrix $Y$ ($Q \in \St{m}{p}$ and $R$ is an $m\times m$ upper triangular matrix).

\subsection{Random test problems}
\label{sec:num_exp:examples}

All random data matrices $A \in \R^{p \times n}$ considered in this section are generated according to a Gaussian distribution, with zero mean and unit variance.

\bigskip
\textbf{Trade-off curves.} Let us first compare the single-unit algorithms, which provide a unit-norm sparse loading vector $z \in \R^n$. We first plot the variance explained by the extracted component against the cardinality of the resulting loading vector $z$. For each algorithm, the sparsity-inducing parameter is incrementally increased to obtain loading vectors $z$ with a cardinality that decreases from $n$ to $1$. The results displayed in Figure \ref{fig:Fig1} are averages of computations on 100 random matrices with dimensions $p=100$ and $n=300$. The considered sparse PCA methods aggregate in two groups: $\Asingle$, $\Asinglec$, $\Greedy$ and $\rSVDc$ outperform the $\SPCA$ and the $\rSVD$ approaches. It seems that these latter methods perform worse because of the $\ell_1$ penalty term used in them. If one, however, post-processes the active part of $z$ according to (\ref{eq:solution_postprocessing_single_unit}), as we do in $\Asingle$, all sparse PCA methods reach the same performance.

\begin{figure}[h!]
\centerline{\includegraphics[height=\figwidth,keepaspectratio]{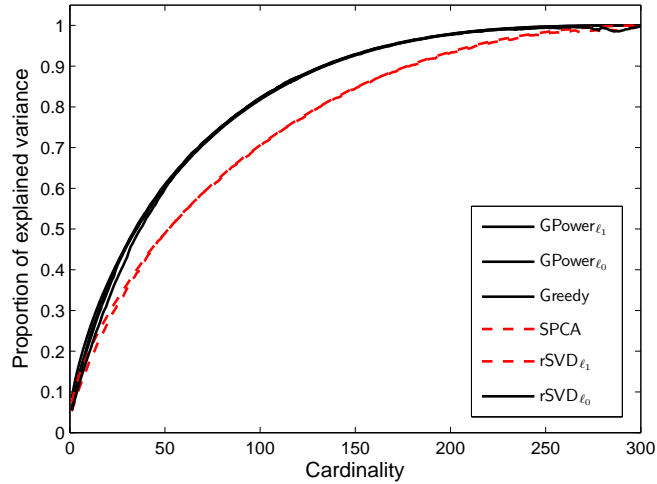}}
\caption{\capsize \textbf{Trade-off curves} between explained variance and cardinality. The vertical axis is the ratio $\Var(z_{\mathrm{sPCA}})/\Var(z_{\mathrm{PCA}})$, where the loading vector $z_{\mathrm{sPCA}}$ is computed by sparse PCA and $z_{\mathrm{PCA}}$ is the first principal loading vector. The considered algorithms aggregate in two groups: $\Asingle$, $\Asinglec$, $\Greedy$ and $\rSVDc$ (top curve), and $\SPCA$ and $\rSVD$ (bottom curve). For a fixed cardinality value, the methods of the first group explain more variance. Postprocessing algorithms $\SPCA$ and $\rSVD$ with equation (\ref{eq:solution_postprocessing_single_unit}), results, however, in the same performance as the other algorithms.}
\label{fig:Fig1}
\end{figure}

\bigskip
\textbf{Controlling sparsity with $\sparsity$.} Among the considered methods, the greedy approach is the only one to directly control the cardinality of the solution, i.e., the desired cardinality is an input of the algorithm. The other methods require a parameter controlling the trade-off between variance and cardinality. Increasing this parameter leads to solutions with smaller cardinality, but the resulting number of nonzero elements can not be precisely predicted. In Figure \ref{fig:Fig2}, we plot the average relationship between the parameter $\sparsity$ and the resulting cardinality of the loading vector $z$ for the two algorithms $\Asingle$ and $\Asinglec$. In view of (\ref{eq:singleL1_sparsity}) (resp. (\ref{eq:singleL0_sparsity})), the entries $i$ of the loading vector $z$ obtained by the $\Asingle$ algorithm (resp. the $\Asinglec$ algorithm) satisfying
\begin{equation}\label{eq:gamma_cutoff}\|a_i\|_2 \leq \sparsity \quad\text{ (resp. } \|a_i\|_2^2 \leq \sparsity) \end{equation} have to be zero. Taking into account the distribution of the norms of the columns of $A$, this provides for every $\sparsity$ a theoretical upper bound on the expected cardinality of the resulting vector $z$.

\begin{figure}[h!]
\centerline{\includegraphics[height=\figwidth,keepaspectratio]{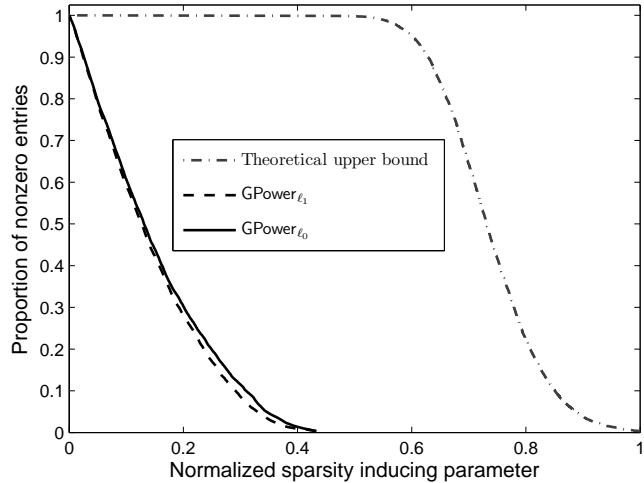}}
\caption{\capsize Dependence of cardinality on the value of the sparsity-inducing parameter $\sparsity$. In case of the $\Asingle$ algorithm, the horizontal axis shows $\sparsity/\|a_{i^*}\|_2$, whereas for the $\Asinglec$ algorithm, we use $\sqrt{\sparsity}/\|a_{i^*}\|_2$. The theoretical upper bound is therefor identical for both methods. The plots are averages based on 100 test problems of size $p=100$ and $n=300$.} \label{fig:Fig2}
\end{figure}

\bigskip
\textbf{Greedy versus the rest.} The considered sparse PCA methods feature different empirical computational complexities. In Figure \ref{fig:Fig3}, we display the average time required by the sparse PCA algorithms to extract one sparse component from Gaussian matrices of dimensions $p=100$ and $n=300$. One immediately notices that the greedy method
slows down significantly as cardinality increases, whereas the speed of the other considered algorithms does not depend on cardinality. Since on average $\Greedy$ is much slower than the other methods, even for low cardinalities, we discard it from all following numerical experiments.

\begin{figure}[h!]
\centerline{\includegraphics[height=\figwidth,keepaspectratio]{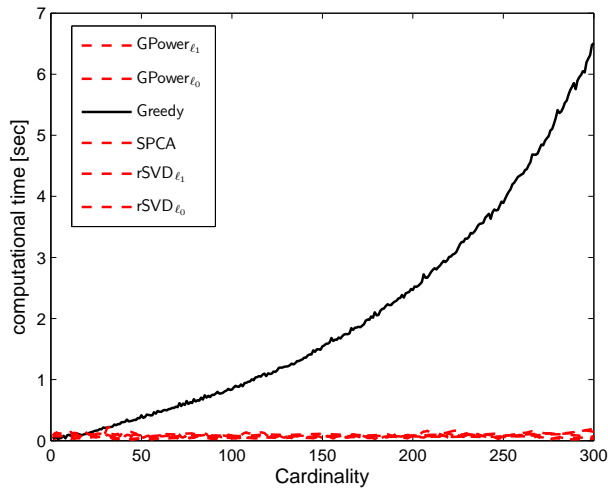}}
\caption{\capsize
The computational complexity of $\Greedy$ grows significantly if it is set out to output a loading vector of increasing cardinality. The speed of the other methods is unaffected by the cardinality target.} \label{fig:Fig3}
\end{figure}

\bigskip
\textbf{Speed and scaling test.} In Tables \ref{tbl:single_unit_time1} and \ref{tbl:single_unit_time2} we compare the speed of the remaining algorithms. Table \ref{tbl:single_unit_time1} deals with problems with a fixed aspect ratio $n/p = 10$, whereas in Table \ref{tbl:single_unit_time2},  $p$ is fixed at 500, and exponentially increasing values of $n$ are considered. For the $\Asingle$ method, the sparsity inducing parameter $\sparsity$ was set to $10\%$ of the upper bound $\sparsity_{\max} = \|a_{i^*}\|_2$. For the $\Asinglec$ method, $\sparsity$ was set to $1\%$ of $\sparsity_{\max} = \|a_{i^*}\|_2^2$ in order to aim for solutions of comparable cardinalities (see (\ref{eq:gamma_cutoff})). These two parameters have also been used for the $\rSVD$ and the $\rSVDc$ methods, respectively. Concerning $\SPCA$, the sparsity parameter has been chosen by trial and error to get, on average, solutions with similar cardinalities as obtained by the other methods. The values displayed in Tables \ref{tbl:single_unit_time1} and \ref{tbl:single_unit_time2} correspond to the average running times of the algorithms on 100 test instances for each problem size. In both tables, the new methods $\Asingle$ and $\Asinglec$ are the fastest. The difference in speed between $\Asingle$ and $\Asinglec$ results from different approaches to fill the active part of $z$: $\Asingle$ requires to compute a rank-one approximation of a submatrix of $A$ (see Equation (\ref{eq:solution_postprocessing_single_unit})), whereas the explicit solution (\ref{eq:spca_single_l0_2}) is available to $\Asinglec$. The linear complexity of the algorithms in the problem size $n$ is clearly visible in Table \ref{tbl:single_unit_time2}.

\begin{table}[h!] \tabsize\centerline{
 \begin{tabular}{l c c c c c }
\hline
$p \times n$&  $100 \times 1000$& $250 \times 2500$ &  $500 \times 5000$ &$750 \times 7500$ & $1000 \times 10000$\\
\hline
$\Asingle$  &    0.10   & 0.86   & 2.45  &4.28& 5.86\\
$\Asinglec$ &      0.03 &   0.42   & 1.21 & 2.07 & 2.85\\
$\SPCA$ &    0.24  &  2.92  & 14.5 & 40.7 & 82.2\\
$\rSVD$   & 0.21  &  1.45 &   6.70& 17.9 & 39.7\\
$\rSVDc$  &  0.20  &  1.33   & 6.06& 15.7 & 35.2\\
 \hline
\end{tabular}}
 \caption{\capsize Average computational time for the extraction of one component (in seconds).} \label{tbl:single_unit_time1}
\end{table}

\begin{table}[h!] \tabsize \centerline{
 \begin{tabular}{l c c c c c c}
\hline
$p \times n$&   $500 \times 1000$ &  $500 \times 2000$ &$500 \times 4000$ & $500 \times 8000$ & $500 \times 16000$\\
\hline
$\Asingle$   &  0.42  &  0.92  &  2.00  &  4.00  &  8.54 \\
$\Asinglec$   &  0.18   & 0.42 &   0.96   & 2.14 &   4.55 \\
$\SPCA$ &      5.20  &  7.20 &  12.0  & 22.6 &  44.7 \\
$\rSVD$   &   1.20    &2.53  &  5.33  & 11.3 &  26.7 \\
$\rSVDc$  &  1.09  &  2.26  &  4.85 &  10.5  & 24.6
\\
 \hline
\end{tabular}}
 \caption{\capsize Average computational time for the extraction of one component (in seconds).} \label{tbl:single_unit_time2}
\end{table}

\bigskip
\textbf{Different convergence mechanisms.} Figure \ref{fig:Fig4} illustrates how the trade-off between explained variance and sparsity evolves in the time of computation for the two methods $\Asingle$ and $\rSVD$. In case of the $\Asingle$ algorithm, the initialization point (\ref{eq:initial_point}) provides a good approximation of the final cardinality. This method then  works on maximizing the variance while keeping the sparsity at a low level throughout. The $\rSVD$ algorithm, in contrast, works in two steps. First, it maximizes the variance, without enforcing sparsity. This corresponds to computing the first principal component and requires thus a first run of the algorithm with random initialization and a sparsity inducing parameter set at zero. In the second run, this parameter is set to a positive value and the method works to rapidly decrease cardinality at the expense of only a modest decrease in explained variance. So, the new algorithm $\Asingle$ performs faster primarily because it combines the two phases into one, simultaneously optimizing the trade-off between variance and sparsity.
\begin{figure}[h!]
\centerline{\includegraphics[height=\figwidth,keepaspectratio]{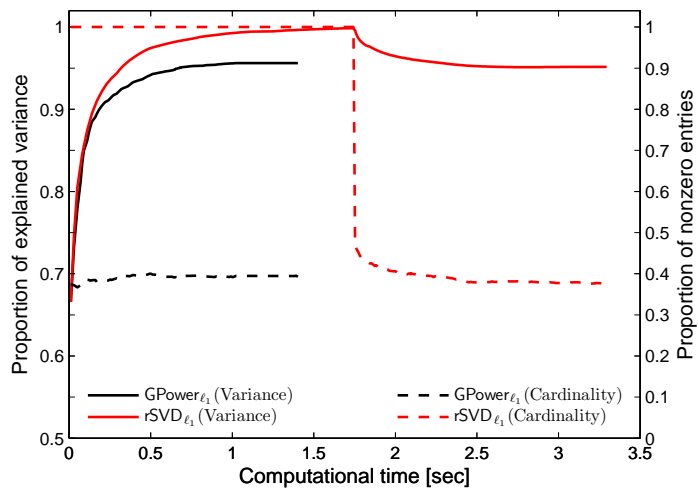}}
\caption{\capsize Evolution of the variance (solid lines and left axis) and cardinality (dashed lines and right axis) in time of computation for the methods $\Asingle$ and $\rSVD$ on a test problem with $p=250$ and $n=2500$. The vertical axis is the ratio $\Var(z_{\mathrm{sPCA}})/\Var(z_{\mathrm{PCA}})$, where the loading vector $z_{\mathrm{sPCA}}$ is computed by sparse PCA and $z_{\mathrm{PCA}}$ is the first principal loading vector. The $\rSVD$ algorithm first solves unconstrained PCA, whereas $\Asingle$ immediately optimizes the trade-off between variance and sparsity.} \label{fig:Fig4}
\end{figure}

\bigskip
\textbf{Extracting more components.} Similar numerical experiments, which include the methods
$\Ablock$ and $\Ablockc$, have been conducted for the extraction of more than one component. A deflation scheme is used by the non-block methods to sequentially compute $m$ components. These experiments lead to similar conclusions as in the single-unit case, i.e, the methods  $\Asingle$, $\Asinglec$, $\Ablock$, $\Ablockc$ and $\rSVDc$ outperform the $\SPCA$ and $\rSVD$ approaches in terms of variance explained at a fixed cardinality. Again, these last two methods can be improved by postprocessing the resulting loading vectors with Algorithm \ref{algo:alternate_optim}, as it is done for $\Ablock$. The average running times  for problems of various sizes are listed in Table \ref{tbl:block_time}. The new power-like methods are significantly faster on all instances.
\begin{table}[H] \tabsize \centerline{ 
 \begin{tabular}{l c c c c c}
\hline
$p \times n$& $50 \times 500$ &  $100 \times 1000$& $250 \times 2500$ &
$500 \times 5000$ &$750 \times 7500$\\
\hline
$\Asingle$   &     0.22  &  0.56  &  4.62 &  12.6& 20.4\\
$\Asinglec$   &   0.06  &  0.17  &  2.15 &   6.16 & 10.3\\
$\Ablock$   &   0.09  &  0.28  &  3.50  & 12.4&23.0\\
$\Ablockc$   &  0.05  &  0.14  &  2.39  &  7.7& 12.4\\
$\SPCA$   &  0.61   & 1.47 &  13.4 &  48.3&113.3\\
$\rSVD$   &  0.30 &   1.15  &  7.92  & 37.4& 97.4\\
$\rSVDc$   &  0.28 &   1.10  &  7.54 &  34.7& 85.7\\
 \hline
\end{tabular}}
 \caption{\capsize Average computational time for the extraction of $m=5$
components (in seconds).} \label{tbl:block_time}
\end{table}

\subsection{Analysis of gene expression data}
\label{sec:num_exp:genes}
Gene expression data results from DNA microarrays and provide the
expression level of thousands of genes across several hundreds of
experiments. The interpretation of these huge databases remains a
challenge. Of particular interest is the identification
of genes that are systematically coexpressed under similar
experimental conditions. We refer to \combib{Riva et al. }\citet{Riva05} and
references therein for more details on microarrays and gene expression data. PCA has been intensively applied in this
context (e.g., \combib{Alter at al. }\citet{Alter03}). Further methods for dimension reduction, such as independent component analysis (\combib{Liebermeister }\citet{Liebermeister02}) or nonnegative matrix factorization (\combib{Brunet et al. }\citet{Brunet04}), have also been used on gene expression data. Sparse PCA, which extracts components involving a few
genes only, is expected to enhance interpretation.

\bigskip
\textbf{Data sets.} The results below focus on four major data sets related to breast
cancer. They are briefly detailed in Table \ref{tbl:cancer_cohorts}. Each sparse PCA algorithm computes ten components from these data sets.
\begin{table}[h] \tabsize \centerline{
 \begin{tabular}{l c c c c c }
\hline
Study & Samples ($p$) & Genes ($n$) & Reference \\
\hline Vijver & 295 & 13319 & \combib{van de Vijver et al. }\citet{vandeVijver2002}\\
Wang & 285 & 14913&  \combib{Wang et al. }\citet{Wang2005}\\
Naderi & 135 & 8278& \combib{Naderi et al. }\citet{Naderi2006}\\
JRH-2 & 101 & 14223& \combib{Sotiriou et al. }\citet{Sotiriou2006}\\
 \hline
\end{tabular}}
 \caption{\capsize Breast cancer cohorts.} \label{tbl:cancer_cohorts}
\end{table}

\bigskip
\textbf{Speed.} The average computational time required by the sparse PCA algorithms on each data set is displayed in Table \ref{tbl:cancer_cohorts_time}. The indicated times are averages on all the computations performed to obtain cardinality ranging from $n$ down to 1.

\begin{table}[h!] \tabsize \centerline{
 \begin{tabular}{l c c c c }
\hline
  & Vijver & Wang & Naderi & JRH-2 \\
\hline
$\Asingle$   &  7.72  &  6.96  &  2.15   & 2.69\\
$\Asinglec$   &  3.80&    4.07 &   1.33  &  1.73\\
$\Ablock$  & 5.40  &  4.37  &  1.77 &  1.14\\
$\Ablockc$  &  5.61 &   7.21  &  2.25   & 1.47\\
$\SPCA$  & 77.7  & 82.1  & 26.7 &  11.2\\
$\rSVD$  &  46.4 &  49.3  & 13.8 &  15.7\\
$\rSVDc$  & 46.8 &  48.4 &  13.7  & 16.5\\
 \hline
\end{tabular}}
 \caption{\capsize Average computational times (in seconds).}
\label{tbl:cancer_cohorts_time}
\end{table}

\bigskip
\textbf{Trade-off curves.} Figure \ref{fig:Fig5} plots the proportion of adjusted variance
versus the cardinality for the ``Vijver'' data set. The other data
sets have similar plots. As for the random test problems, this performance criterion does not discriminate among the different algorithms. All methods have in fact the same performance, provided that the $\SPCA$ and $\rSVD$ approaches are used with postprocessing by Algorithm \ref{algo:alternate_optim}.
\begin{figure}[h]
\centerline{\includegraphics[height=\figwidth,keepaspectratio]{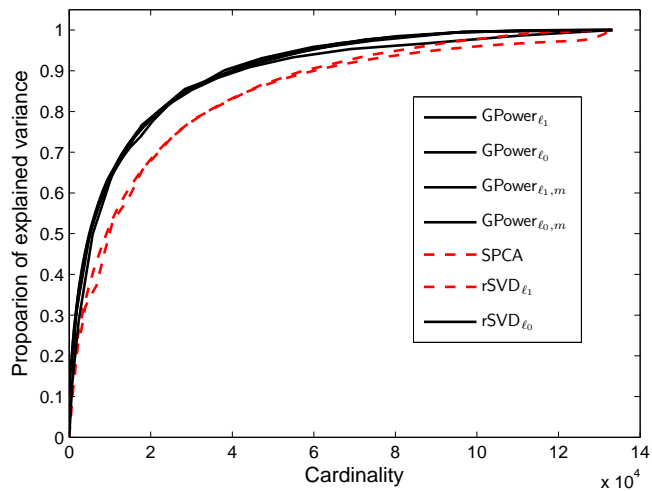}}
\caption{\capsize
\textbf{Trade-off curves} between explained variance and cardinality (case
of the ``Vijver'' data). The vertical axis is the ratio
$\AVar(Z_{\mathrm{sPCA}})/\AVar(Z_{\mathrm{PCA}})$, where the loading
vectors $Z_{\mathrm{sPCA}}$ are computed by sparse PCA and
$Z_{\mathrm{PCA}}$ are the $m$ first principal loading
vectors.\label{fig:Fig5}}
\end{figure}

\bigskip
\textbf{Interpretability.} A more interesting performance criterion is to estimate the biological interpretability of the extracted components. The \emph{pathway
enrichment index} (PEI) proposed by \combib{Teschendorff et al. }\citet{Teschendorff06} measures
the statistical significance of the overlap between two kinds of
gene sets. The first sets are inferred from the computed components
by retaining the most expressed genes, whereas the second sets
result from biological knowledge. For instance, metabolic pathways
provide sets of genes known to participate together when a certain
biological function is required. An alternative is given by the
regulatory motifs: genes tagged with an identical motif are likely
to be coexpressed. One expects sparse PCA methods to recover some of
these biologically significant sets. Table \ref{tbl:PEI1} displays the
PEI based on 536 metabolic pathways related to cancer. The PEI is the fraction
of these 536 sets presenting a statistically significant overlap
with the genes inferred from the sparse principal components. The
values in Table \ref{tbl:PEI1} correspond to the largest PEI obtained
among all possible cardinalities. Similarly, Table \ref{tbl:PEI2} is
based on 173 motifs. More details on the selected pathways and
motifs can be found in \combib{Teschendorff et al. }\citet{Teschendorff06}. This analysis clearly indicates that the sparse PCA methods perform much better than PCA in this context. Furthermore, the new $\mathsf{GPower}$ algorithms, and especially the block formulations, provide largest
PEI values for both types of biological information. In terms of biological interpretability, they systematically outperform previously published algorithms.

\begin{table}[h!]  \tabsize \centerline{
 \begin{tabular}{l c c c c c }
\hline
 & Vijver & Wang & Naderi & JRH-2 \\
\hline
PCA & 0.0728 & 0.0466 & 0.0149 & 0.0690\\
$\Asingle$  &\textbf{0.1493}    &0.1026    &0.0728    &0.1250\\
$\Asingle$  &0.1250    &0.1250    &0.0672    &0.1026\\
$\Ablock$     &0.1418    &0.1250    &\textbf{0.1026}    &\textbf{0.1381}\\
$\Ablockc$   &0.1362    &\textbf{0.1287 }   &0.1007    &0.1250\\
$\SPCA$    &0.1362    &0.1007    &0.0840    &0.1007\\
$\rSVD$    &0.1213    &0.1175    &0.0914    &0.0914\\
$\rSVDc$    &0.1175    &0.0970    &0.0634    &0.1063\\
 \hline
\end{tabular}}
 \caption{\capsize PEI-values based on a set of 536 cancer-related pathways.} \label{tbl:PEI1}
\end{table}

\begin{table}[h!]  \tabsize \centerline{
 \begin{tabular}{l c c c c c }
\hline
 & Vijver & Wang & Naderi & JRH-2 \\
\hline
$\PCA$ & 0.0347 & 0 & 0.0289 & 0.0405\\
$\Asingle$   &0.1850   & 0.0867   & 0.0983   & 0.1792\\
$\Asinglec$     &0.1676   & 0.0809   & 0.0925   & \textbf{0.1908}\\
$\Ablock$    &\textbf{0.1908}   & \textbf{0.1156}   & \textbf{0.1329}   & 0.1850 \\
$\Ablockc$    &0.1850   & 0.1098   & \textbf{0.1329}   & 0.1734\\
$\SPCA$    &0.1734   & 0.0925   & 0.0809   & 0.1214\\
$\rSVD$   &0.1387   & 0.0809   & 0.1214   & 0.1503\\
$\rSVDc$   &0.1445   & 0.0867   & 0.0867   & 0.1850\\
 \hline
\end{tabular}}
 \caption{\capsize PEI-values based on a set of 173 motif-regulatory gene sets.} \label{tbl:PEI2}
\end{table}

\section{Conclusion}

We have proposed two single-unit and two block formulations of the sparse PCA problem and constructed reformulations with several favorable properties. First, the reformulated problems are of the form of maximization of a convex function on a compact set, with the feasible set being either a unit Euclidean sphere or the Stiefel manifold. This structure allows for the design and iteration complexity analysis of a simple gradient scheme which applied to our sparse PCA setting results in four new algorithms for computing sparse principal components of a matrix $A\in \R^{p\times n}$. Second, our algorithms appear to be faster if either the objective function or the feasible set are strongly convex, which holds in the single-unit case and can be enforced in the block case. Third, the dimension of the feasible sets does not depend on $n$ but on $p$ and on the number $m$ of components to be extracted. This is a highly desirable property if $p\ll n$. Last but not least, on random and real-life biological data, our methods systematically outperform the existing algorithms both in speed and trade-off performance. Finally, in the case of the biological data, the components obtained by our block algorithms deliver the richest biological interpretation as compared to the components extracted by the other methods.

\subsection*{Acknowlegments}
This paper presents research results of the Belgian Network DYSCO (Dynamical
Systems, Control, and Optimization), funded by the Interuniversity
Attraction Poles Programme, initiated by the Belgian State,
Science Policy Office. The scientific responsibility rests with
its authors. Research of Yurii Nesterov and Peter Richtárik has been supported by the grant ``Action de recherche concert\'ee ARC 04/09-315'' from the ``Direction de la recherche scientifique - Communaut\'e fran\c{c}aise de Belgique''. Michel Journ\'{e}e is a research fellow of the
Belgian National Fund for Scientific Research (FNRS).

\section{Appendix A}

In this appendix we characterize a class of functions with strongly convex level sets. First we need to collect some basic preliminary facts.  All the inequalities of Proposition \ref{prop:appendix} are well-known in the literature.

\begin{prpstn}\label{prop:appendix}\begin{itemize}
\item[(i)] If $f$ is a  strongly convex function with convexity parameter $\sigma_f$, then for all $x,y$ and $0\leq \alpha \leq 1$,
\begin{equation}\label{app:eq:strong_conv}f(\alpha x + (1-\alpha)y) \leq \alpha f(x) + (1-\alpha)f(y) - \frac{\sigma_f}{2}\alpha(1-\alpha)\|x-y\|^2.\end{equation}
\item[(ii)]If $f$ is a convex differentiable function and its gradient is Lipschitz continuous with constant $L_f$, then for all $x$  and $h$,
\begin{equation}\label{app:eq:Lipsch_gradient}f(x+h)\leq f(x) + \ve{f'(x)}{h} + \frac{L_f}{2}\|h\|^2,\end{equation}
and
\begin{equation}\label{app:eq:bound_on_derivative}\|f'(x)\|_* \leq \sqrt{2L_f(f(x)-f_*)},\end{equation}
where $f_* \eqdef \min_{x\in \E} f(x)$.
\end{itemize}
\end{prpstn}

We are now ready for the main result of this section.

\begin{thrm}[Strongly convex level sets] \label{app:thm:level_sets} Let $f:\E\to \R$ be a nonnegative strongly convex function with convexity parameter $\sigma_f>0$. Also assume $f$ has a Lipschitz continuous gradient with Lipschitz constant $L_f>0$. Then for any $\level>0$, the  set
\[\Q_\level  \eqdef \{x\suchthat f(x)\leq \level\}\]
is strongly convex with convexity parameter
\[\sigma_{\Q_\level} = \frac{\sigma_f}{\sqrt{2\level L_f}}.\]
\end{thrm}
\begin{proof} Consider any $x,y\in \Q_\level$, scalar $0\leq \alpha \leq 1$ and let $z_\alpha = \alpha x + (1-\alpha)y$. Notice that by convexity, $f(z_\alpha)\leq \level$. For any $u\in \E$,
\begin{align*}f(z_\alpha + u)
& \substack{(\ref{app:eq:Lipsch_gradient})\\\leq}  f(z_\alpha) + \ve{f'(z_\alpha)}{u} + \frac{L_f}{2}\|u\|^2\\
& \leq f(z_\alpha) + \|f'(z_\alpha)\|\|u\| + \frac{L_f}{2}\|u\|^2 \\
& \substack{(\ref{app:eq:bound_on_derivative})\\\leq} f(z_\alpha) + \sqrt{2L_f f(z_\alpha)}\|u\| + \frac{L_f}{2}\|u\|^2 \\
& = \left(\sqrt{f(z_\alpha)} + \sqrt{\tfrac{L_f}{2}}\|u\|\right)^2\\
& \substack{(\ref{app:eq:strong_conv})\\\leq} \left(\sqrt{\level - \beta} + \sqrt{\tfrac{L_f}{2}}\|u\|\right)^2,
\end{align*}
where \begin{equation}\label{app:eq:beta}\beta = \frac{\sigma_f}{2}\alpha(1-\alpha)\|x-y\|^2.\end{equation}
In view of (\ref{eq:set_strongly_convex}), it remains to show that the last displayed expression is bounded above by $\level$ whenever $u$ is of the form
\begin{equation}\label{app:eq:u}u=\frac{\sigma_{\Q_\level}}{2}\alpha(1-\alpha)\|x-y\|^2s = \frac{\sigma_f}{2\sqrt{2\level L_f}}\alpha(1-\alpha)\|x-y\|^2s,\end{equation} for some $s\in \sph$. However, this follows directly from concavity of the scalar function $g(t)=\sqrt{t}$:
\begin{align*}\sqrt{\level-\beta} = g(\level-\beta) &\leq g(\level) - \ve{g'(\level)}{\beta}\\
&= \sqrt{\level} - \frac{\beta}{2\sqrt{\level}}\\
& \substack{(\ref{app:eq:beta})\\\leq} \sqrt{\level} - \frac{\sigma_f}{4\sqrt{\level}}\alpha(1-\alpha)\|x-y\|^2\\
&\substack{(\ref{app:eq:u})\\\leq} \sqrt{\level} - \sqrt{\frac{L_f}{2}} \|u\|.
\end{align*}

\end{proof}

\begin{xmpl}Let $f(x)=\|x\|^2$. Note that $\sigma_f=L_f=2$. If we let $\level=r^2$, then
\[\Q_{\level} = \{x \suchthat f(x)\leq \level\} = \{x \suchthat \|x\|\leq r\} = r \cdot \ball.\]
We have shown before (see the discussion immediately following Assumption~\ref{ass:3}), that the strong convexity parameter of this set is
$\sigma_{\Q_\level} = \tfrac{1}{r}$. Note that we recover this as a special case of Theorem~\ref{app:thm:level_sets}:
\[\sigma_{\Q_\level} = \frac{\sigma_f}{\sqrt{2\level L_f}} = \frac{1}{r}.\]
\end{xmpl}

\vskip 0.2in

\end{document}